\documentclass[reqno,fleqn,twoside,a4paper,12pt]{amsart}
\usepackage[utf8]{inputenc}
\usepackage{hyperref}
\usepackage[a4paper,left=25mm,right=25mm,top=50mm,bottom=50mm]{geometry}
\usepackage{amsmath}
\usepackage{amssymb}
\usepackage{amsfonts}
\usepackage[usenames]{color}
\usepackage{cite}
\usepackage{enumerate}


\numberwithin{equation}{section}
\newtheorem{theorem}{Theorem}

\newtheorem{corollary}[theorem]{Corollary}

\newtheorem{definition}[theorem]{Definition}
\newtheorem{example}[theorem]{Example}

\newtheorem{lemma}[theorem]{Lemma}

\newtheorem{remark}[theorem]{Remark}

\def\F{\leavevmode\setbox0=\hbox{F}\kern0pt\rlap
	{\kern.04em\raise.188\ht0\hbox{-}}F}


\begin{document}
	
		\address{$^{\dagger }$Department of Mathematics,  Faculty of Science and Arts, Mu\c{s} Alparslan University, Mu\c{s} 49100, Turkey}
		\email{\href{mailto:isikhuseyin76@gmail.com}{isikhuseyin76@gmail.com}}
		
		\title[A new class of set-valued contractions]{\sc Fractional differential inclusions with a new class of set-valued contractions}
		\author[H. I\c{S}IK] {H\"{u}sey\.{I}n I\c{s}\i k$^{\dagger }$}
	
	\begin{abstract}
	The aim of this study to investigate the existence of solutions for the following nonlocal integral boundary value problem of Caputo
	type fractional differential inclusions:
	\begin{equation*}
	\left\{
	\begin{array}{ll} 
	{^C}D^\beta_{t_0}x(t)\in F(t,x(t)), & t\in J=[t_0,T], \ n-1<\beta<n, \vspace{0,3cm} \\
	x^{(k)}(\alpha)=a_k+\int_{t_0}^{\alpha}g_k(s,x(s))ds, & k=0,1,\ldots,n-1, \ \alpha \in (t_0,T),
	\end{array}%
	\right.  
	\end{equation*}
	where $F\colon J \times \mathbb{R} \to P(\mathbb{R})$ is a multivalued map, $P(\mathbb{R})$ is the family of all nonempty subsets of $\mathbb{R}, \ g_k \colon J \times \mathbb{R} \to \mathbb{R}$ is a given continuous function, $a_k \in \mathbb{R}$ and ${^C}D^\beta_{t_0}$ denotes the Caputo fractional
	derivative of order $\beta, \ n = [\beta] + 1, \ [\beta ]$ denotes the integer part of the real number $\beta.$
	
	To achieve our goals, we take advantage of fixed point theorems for multivalued mappings satisfying a new class of contractive conditions in the setting of complete metric spaces. We derive new fixed point results which extend and improve many results in the literature by means of this new class of contractions. We also supply some examples to support the new theory.
	\end{abstract}

\maketitle
\section{\textbf{Introduction}}
Fixed point theory is one of the most significant and beneficial instruments in mathematical analysis on account of the fact that it purveys sufficient and necessary conditions at finding the existence and uniqueness of a solution of mathematical and practical problems which can be reduced to an equivalent fixed point problem. In particular, Banach contraction principle, in which states that every contraction self-map on a complete metric space has a unique fixed point, has a variety of applications in many disciplines such as chemistry, physics, biology, computer science and many branches of mathematics. This fundamental principle have been generalized in two main directions; either by generalizing the domain of the mapping or by weakening the contractive condition or sometimes even both. Some of those were studied by Berinde \cite{ber}, Chatterja \cite{chat}, \'{C}iri\'{c} \cite{cir,ciric}, Hardy and Rogers \cite{hardy}, Kannan \cite{kannan}, Reich \cite{reich}, Suzuki \cite{suzuki} and Zamfirescu \cite{zam}. In other respects, Nadler \cite{nadler} extended Banach contraction principle from self-maps to multivalued mappings by using the notion of the Hausdorff metric. The theory of multivalued mappings has various applications in optimal control theory, convex optimization, integral inclusions, fractional differential inclusions,  economics and game theory. Recently, Jleli and Samet \cite{jleli} introduced a new type of contractive self-maps known as $\vartheta$-contaction and proved the existence and uniqueness of fixed points for these types of mappings by using a new technique of proof via the properties of the functions $\vartheta$. After then, several researchers extended the results in \cite{jleli} to multivalued mappings in different directions, see for example, Nastasi annd Vetro \cite{nas}, Pansuwan {\it et al}. \cite{pan} and Vetro \cite{vetro}.

In this study, we introduce a new class of contractions for multivalued mappings by weakening the conditions on $\vartheta$ and by using auxilary functions. Using this new type of contractions, we establish fixed point theorems for multivalued mappings on complete metric spaces, which improve and
extend the results in \cite{ber,chat,cir,ciric,hardy,jleli,kannan,nadler,reich,vetro,zam} and many others in the literature. Some examples is constructed in order to illustrate the generality of our results. As an application of the obtained results, sufficient conditions are discussed to ensure the existence of solutions of the following nonlocal integral boundary value problem of Caputo type fractional differential inclusions:
\begin{equation}\label{frac}
\left\{
\begin{array}{ll} 
{^C}D^\beta_{t_0}x(t)\in F(t,x(t)), & t\in J=[t_0,T], \ n-1<\beta<n, \vspace{0,3cm} \\
x^{(k)}(\alpha)=a_k+\int_{t_0}^{\alpha}g_k(s,x(s))ds, & k=0,1,\ldots,n-1, \ \alpha \in (t_0,T),
\end{array}%
\right.  
\end{equation}
where $F\colon J \times \mathbb{R} \to P(\mathbb{R})$ is a multivalued map, $P(\mathbb{R})$ is the family of all nonempty subsets of $\mathbb{R}, \ g_k \colon J \times \mathbb{R} \to \mathbb{R}$ is a given continuous function, $a_k \in \mathbb{R}$ and ${^C}D^\beta_{t_0}$ denotes the Caputo fractional
derivative of order $\beta, \ n = [\beta] + 1, \ [\beta ]$ denotes the integer part of the real number $\beta.$

\section{\textbf{Preliminaries and Background}}
Here, we recollect some basic definitions, lemmas, notations and
some known theorems which are helpful for understanding of this
paper. In the sequel, we will indicate the set of all non-negative real numbers and the set of all natural numbers by the letters $\mathbb{R^+} \mbox{ and } \mathbb{N}$, respectively. Let $ (X,d) $ be a metric space and denote the family of nonempty, closed and bounded subsets of $ X$ by $ CB(X).$ For $ A,B
\in CB(X),$ define $H\colon CB(X) \times CB(X) \to \mathbb{R^+}$ by
\begin{align*}
H(A,B)=\max \left\lbrace \sup_{a\in A} \ d(a,B), \ \sup_{b \in B} \ d(b,A)\right\rbrace
\end{align*}
where $ d(a,B)=\inf\left\lbrace d(a,x)\!: \ x \in B \right\rbrace $.
Such a function $H$ is called the Pompeiu-Hausdorff metric induced
by $d,$ for more details, see \cite{berinde}. Also,
denote the family of nonempty and closed subsets of $X$ by $CL(X)$
and the family of nonempty and compact subsets of $X$ by $K(X)$.
Note that $H\colon CL(X) \times CL(X) \to [0,\infty]$ is a
generalized Pompeiu-Hausdorff metric, that is, $H(A,B)=\infty$ if
$\max \left\lbrace \sup_{a\in A} \ d(a,B), \ \sup_{b \in B} \
d(b,A)\right\rbrace$ does not exist in $\mathbb{R}$.

\begin{lemma} [\cite{vetro}] \label{lm1}
	Let $ (X,d) $ be a metric space and $ A,B \in CL(X)$ with $H(A,B)>0.$ Then, for each $h > 1$ and for each $a \in A,$ there exists $b=b(a) \in B$ such that \linebreak $d(a,b)<hH(A,B).$
\end{lemma}

Following the results in \cite{jleli}, Vetro \cite{vetro} established fixed point results for multivalued mappings.

\begin{definition}[\cite{jleli,vetro}]\rm
	Let $(X,d) $ be a metric space. A map $T\colon X \to CL(X)$ is called a {\it weak $\vartheta$-contraction}, if there exist $k \in (0,1)$ and $\vartheta \in \Theta$ such that
	\begin{equation} \label{theta}
	\vartheta(H(Tx,Ty)) \leq [\vartheta(d(x,y))]^{k},
	\end{equation}
	for all $x,y \in X$ with $H(Tx,Ty)>0$, where $\Theta$ is the set of
	functions $\vartheta\colon (0,\infty) \to (1,\infty)$ satisfying the
	following conditions:
	\begin{enumerate}
		\item[$(\vartheta1)$] $\vartheta $ is non-decreasing; 
		\item[$(\vartheta2)$] for each sequence $\{t_{n}\} \subset (0,\infty),$ \ \ $\lim_{n \to \infty} \ \vartheta (t_{n})=1$ if and only if $\lim_{n \to \infty} \ t_{n}=0;$
		\item[$(\vartheta3)$] there exist $r \in (0,1)$ and $\lambda \in (0,\infty]$ such that $\lim_{t \to 0^{+}} \dfrac{\vartheta (t)-1}{t^{r}}=\lambda.$
	\end{enumerate}
\end{definition}

The following functions $\vartheta_{i}\colon (0,\infty) \to
(1,\infty)$ for $i \in \left\lbrace1,2 \right\rbrace ,$ are the
elements of $\Theta.$ Furthermore, substituting in (\ref{theta})
these functions, we obtain some contractions known in the
literature: for all $x,y \in X$ with $H(Tx,Ty)>0,$
\begin{align*}
\vartheta_{1}&(t)=e^{\sqrt{t}}, \ \ \ \ \ \ \ \ \ H(Tx,Ty) \le k^{2}d(x,y), \\
\vartheta_{2}&(t)=e^{\sqrt{te^{t}}}, \ \ \ \ \ \ \  \frac{H(Tx,Ty)}{d(x,y)}e^{H(Tx,Ty)-d(x,y)} \le k^{2}.
\end{align*}

\begin{theorem} [\cite{vetro}] \label{thvetro}
	Let $(X,d)$ be a complete metric space and $T\colon X \to K(X)$ be a weak $\vartheta$-contraction. Then $T$ has a fixed point, that is, there exists a point $u \in X$ such that $u \in Tu.$
\end{theorem}

Note that Theorem \ref{thvetro} is invalid, if we take $CB(X)$
instead of $K(X)$. In \cite{vetro}, Vetro showed that
Theorem \ref{thvetro} is still true for $T\colon X \to CB(X),$
whenever $\vartheta \in \Theta$ is right continuous.

We will not be need the condition $(\vartheta2)$ in our results. Thence, we denote by $\Omega$ the set of all fuctions $\vartheta$ satisfying the conditions $(\vartheta1)$ and $(\vartheta3).$ We can define the functions which belong to the set $\Omega$ but not to $\Theta$ as shown in the following examples.

\begin{example}
Define $\vartheta\colon (0,\infty) \to (1,\infty)$ with $\vartheta(t)=e^{\sqrt{t+1}}.$ Evidently $\vartheta$ satisfies $(\vartheta1)$ and since $\lim_{t \to 0^{+}} \, (e^{\sqrt{t+1}}-1)/t^{r}=\infty$ for $r \in (0,1),$ also $(\vartheta3).$ However, $\vartheta$ doesn't satisfy the condition $(\vartheta2).$ Indeed, consider $t_{n}=\frac{1}{n}$ for all $ n \in \mathbb{N},$ then $\lim_{n \to \infty} \ t_{n}=0$ and
$\lim_{n \to \infty} \ \vartheta (t_{n})=e\neq 1.$ Consequently, $\vartheta \in \Omega$ while $\vartheta \notin \Theta.$
\end{example}

\begin{example}
	Let $a>1$ and $\vartheta(t)=a+\ln(\sqrt{t+1}).$ It can easily be seen that $\vartheta$ satisfies the conditions $(\vartheta1)$ and $(\vartheta3).$ But if we take $t_{n}=\frac{1}{n}$ for all $ n \in \mathbb{N},$ then $\lim_{n \to \infty} \, t_{n}=0$ and
	$\lim_{n \to \infty} \ \vartheta (t_{n})=a>1.$ Hence, $\vartheta \in \Omega$ and $\vartheta \notin \Theta.$
\end{example}
The next lemma will help us to make up for the lack of the condition $(\vartheta2)$ in the proofs.
\begin{lemma} \label{l1}
Let $\vartheta\colon (0,\infty) \to (1,\infty)$ be a non-decreasing function and $\left\lbrace t_{n} \right\rbrace \subset (0,\infty)$ a decreasing sequence such that $\lim_{n \to \infty} \ \vartheta (t_{n})=1.$ Then, we have $\lim_{n \to \infty} \ t_{n}=0.$
\end{lemma}

\begin{proof}
Since the sequence $\left\lbrace t_{n} \right\rbrace$ is decreasing, there exists $t \ge 0$ such that $\lim_{n \to \infty}  t_{n}=t.$ Suppose that $t>0.$ Considering the fact that $\vartheta $ is non-decreasing and $t_{n} \ge t,$ we get $\vartheta(t_{n}) \ge \vartheta(t),$ for all $n \ge 0.$ Taking the limit as $n \to \infty$ in the last inequality, we deduce $1=\lim_{n \to \infty} \ \vartheta(t_{n}) \ge \vartheta(t)$ which contradicts by the definition of $\vartheta, $ hence $t=0.$
\end{proof}

Now, following the lines in \cite{cons}, we denote by $\mathcal{P}$ the set of all continuous mappings \linebreak $\rho \colon (\mathbb{R^+})^5 \to \mathbb{R^+} $ satisfying the following conditions:
\begin{enumerate}
	\item[$(\rho1)$] $\rho(1,1,1,2,0),\rho(1,1,1,0,2),\rho(1,1,1,1,1) \in (0,1];$ 
	\item[$(\rho2)$] $\rho$ is sub-homogeneous, that is, for all $(x_{1},x_{2},x_{3},x_{4},x_{5}) \in (\mathbb{R^+})^5 \mbox{ and } \alpha \ge 0,$ we have $\rho(\alpha x_{1},\alpha x_{2},\alpha x_{3},\alpha x_{4}, \alpha x_{5})\le \alpha \rho( x_{1},x_{2},x_{3},x_{4},x_{5});$
	\item[$(\rho3)$] $\rho$ is a non-decreasing function, that is, for $x_i, y_i \in\mathbb{R^+}, \ x_i \le y_i, \ i = 1,\ldots,5,$ we have \[\rho( x_{1}, x_{2}, x_{3}, x_{4},  x_{5})\le  \rho( y_{1},y_{2},y_{3},y_{4},y_{5})\] and if $x_i, y_i \in\mathbb{R^+}, \ x_i < y_i, \ i = 1,\ldots,4,$ then
	\[\rho( x_{1}, x_{2}, x_{3}, x_{4},  0)<  \rho( y_{1},y_{2},y_{3},y_{4},0) \mbox{  and  } \rho( x_{1}, x_{2}, x_{3},0, x_{4},)<  \rho( y_{1},y_{2},y_{3},0,y_{4}).\]
\end{enumerate} 
Then we have the next result.
\begin{lemma} \label{l2}
If $\rho \in \mathcal{P} \mbox{ and } u,v\in\mathbb{R^+}$ are such that \[u < \max\left\lbrace \rho(v,v,u,v+u,0),\rho(v,v,u,0,v+u),\rho(v,u,v,v+u,0),\rho(v,u,v,0,v+u) \right\rbrace, \]
then $u < v.$
\end{lemma}
\begin{proof}
	Without loss of generality, we can suppose that $u < \rho(v,v,u,v+u,0).$ If $v\le u,$ then
	\[u < \rho(v,v,u,v+u,0)\le \rho(u,u,u,2u,0) \le u\rho(1,1,1,2,0) \le u \]
which is a contradiction. Thus, we deduce that $u < v.$
\end{proof}
We are now ready to give the following definition.
\begin{definition}
	Let $(X,d) $ be a metric space. A multivalued mapping $T\colon X \to CL(X)$ is called a {\it $\vartheta_\rho$-contraction}, if there exist $\vartheta \in \Omega, \ \rho \in \mathcal{P}  \mbox{ and } k\in (0,1)$ such that
	\begin{equation} \label{a1}
	\vartheta(H(Tx,Ty)) \leq [\vartheta(\rho(d(x,y),d(x,Tx),d(y,Ty),d(x,Ty),d(y,Tx)))]^{k},
	\end{equation}
	for all $x,y \in X$ with $H(Tx,Ty)>0$.
\end{definition}

\begin{remark} \label{rem}
	Let $(X,d)$ be a metric space. If $T\colon X \to CL(X)$ is a {\it 
		$\vartheta_\rho$-contraction}, then by (\ref{a1}), we get
	\begin{align*}
	\ln\vartheta(H(Tx,Ty))&\le k \ln\vartheta(\rho(d(x,y),d(x,Tx),d(y,Ty),d(x,Ty),d(y,Tx))) \\
	&<\ln\vartheta(\rho(d(x,y),d(x,Tx),d(y,Ty),d(x,Ty),d(y,Tx))).
	\end{align*}
	Since $\vartheta$ is non-decreasing, we obtain 
	\begin{equation*}
	H(Tx,Ty) < \rho(d(x,y),d(x,Tx),d(y,Ty),d(x,Ty),d(y,Tx)),
	\end{equation*}
	for all $x,y \in X$ with $Tx \neq Ty.$ This implies that 
	\begin{equation*}
	H(Tx,Ty) \le \rho(d(x,y),d(x,Tx),d(y,Ty),d(x,Ty),d(y,Tx)), \mbox{ for all } x,y \in X.
	\end{equation*}
\end{remark}
\section{\textbf{Main Results}}
The first result of this study is as follows.
\begin{theorem}\label{t1}
	Let $(X,d)$ be a complete metric space and $T\colon X \to K(X)$ a 
	$\vartheta_\rho$-contraction. Then $T$ has a fixed point.
\end{theorem}

\begin{proof}
	Let $x_{0}$ be an arbitrary point of $X$ and $x_{1}\in Tx_{0}.$ If $x_{0}=x_{1}$ or $x_{1}\in Tx_{1},$ then $x_{1}$ is a fixed point of $T$ and so the proof is completed. Because of this, assume that $x_{0}\neq x_{1}$ and $x_{1}\notin Tx_{1},$ then $d(x_{1},Tx_{1})>0$ and hence $H(Tx_{0},Tx_{1})>0.$ Since $Tx_{1}$ is compact, there exists $x_{2}\in Tx_{1}$ such that $d(x_{1},x_{2})=d(x_{1},Tx_{1}).$ Bearing in mind that the functions $\vartheta \mbox{ and } \rho$ are non-decreasing, by \eqref{a1}, we have
	\begin{align} \label{a2}
	\vartheta(d(x_{1},x_{2}))&=\vartheta(d(x_{1},Tx_{1})) \leq \vartheta(H(Tx_{0},Tx_{1})) \nonumber \\
	&\leq [\vartheta(\rho(d(x_{0},x_{1}),d(x_{0},Tx_{0}),d(x_{1},Tx_{1}),d(x_{0},Tx_{1}),d(x_{1},Tx_{0})))]^{k}\nonumber\\
	&\le[\vartheta(\rho(d(x_{0},x_{1}),d(x_{0},x_{1}),d(x_{1},x_{2}),d(x_{0},x_{1})+d(x_{1},x_{2}),0))]^{k}.
	\end{align}
	By Remark \ref{rem}, this inequality implies that
	\[d(x_{1},x_{2})< \rho(d(x_{0},x_{1}),d(x_{0},x_{1}),d(x_{1},x_{2}),d(x_{0},x_{1})+d(x_{1},x_{2}),0).\]
	From Lemma \ref{l2}, we get that $d(x_{1},x_{2}) < d(x_{0},x_{1}).$ Thus, using the properties of $\vartheta \mbox{ and } \rho$ in \eqref{a2}, we infer
	\begin{align*}
	\vartheta(d(x_{1},x_{2}))&\le  [\vartheta(\rho(d(x_{0},x_{1}),d(x_{0},x_{1}),d(x_{1},x_{2}),d(x_{0},x_{1})+d(x_{1},x_{2}),0))]^{k} \\
	&<[\vartheta(\rho(d(x_{0},x_{1}),d(x_{0},x_{1}),d(x_{0},x_{1}),2d(x_{0},x_{1}),0))]^{k}\\
	&\le[\vartheta(d(x_{0},x_{1})\rho(1,1,1,2,0))]^{k} \le[\vartheta(d(x_{0},x_{1}))]^{k}.
	\end{align*}
	Following the previous procedures, we can assume that $x_{1}\neq x_{2}$ and $x_{2}\notin Tx_{2}.$ Then $d(x_{2},Tx_{2})>0,$ and so $H(Tx_{1},Tx_{2})>0.$ Since $Tx_{2}$ is compact, there exists $x_{3}\in Tx_{2}$ such that $d(x_{2},x_{3})=d(x_{2},Tx_{2}).$ Considering $(\vartheta1), \ (\rho3)$ and \eqref{a1}, we get
	\begin{align} \label{a3}
	\vartheta(d(x_{2},x_{3}))&=\vartheta(d(x_{2},Tx_{2})) \leq \vartheta(H(Tx_{1},Tx_{2})) \nonumber \\
	&\leq [\vartheta(\rho(d(x_{1},x_{2}),d(x_{1},Tx_{1}),d(x_{2},Tx_{2}),d(x_{1},Tx_{2}),d(x_{2},Tx_{1})))]^{k}\nonumber\\
	&\le[\vartheta(\rho(d(x_{1},x_{2}),d(x_{1},x_{2}),d(x_{2},x_{3}),d(x_{1},x_{2})+d(x_{2},x_{3}),0))]^{k},
	\end{align}
	follows by Remark \ref{rem} that
	\[d(x_{2},x_{3})<\rho(d(x_{1},x_{2}),d(x_{1},x_{2}),d(x_{2},x_{3}),d(x_{1},x_{2})+d(x_{2},x_{3}),0).\]
	Again from Lemma \ref{l2}, we obtain that $d(x_{2},x_{3}) < d(x_{1},x_{2}).$ Thereby, using the properties of $\vartheta \mbox{ and } \rho$ in \eqref{a3}, we deduce
	\begin{align*}
	\vartheta(d(x_{2},x_{3}))&\le  [\vartheta(\rho(d(x_{1},x_{2}),d(x_{1},x_{2}),d(x_{2},x_{3}),d(x_{1},x_{2})+d(x_{2},x_{3}),0))]^{k} \\
	&<[\vartheta(\rho(d(x_{1},x_{2}),d(x_{1},x_{2}),d(x_{1},x_{2}),2d(x_{1},x_{2}),0))]^{k}\\
	&\le[\vartheta(d(x_{1},x_{2})\rho(1,1,1,2,0))]^{k} \le[\vartheta(d(x_{1},x_{2}))]^{k}.
	\end{align*}
	Repeating this process, we can constitute a sequence $\left\lbrace x_{n} \right\rbrace \subset X$ such that $x_{n}\neq x_{n+1} \in Tx_{n}$ and
	\begin{align}\label{3}
	1<\vartheta(d(x_{n},x_{n+1}))<[\vartheta(d(x_{n-1},x_{n}))]^{k},
	\end{align}
	for all $ n \in \mathbb{N}.$ Letting $\sigma_{n}:=d(x_{n},x_{n+1})$ for all $n \in \mathbb{N} \cup \{0\},$ from \eqref{3}, we get
	\begin{equation}\label{6}
	1<\vartheta(\sigma_{n}) < [\vartheta(\sigma_{0})]^{k^{n}}, \ \ \ \ \ \text{for all } n \in \mathbb{N},
	\end{equation}%
	which implies that $\lim_{n\rightarrow \infty }  \vartheta(\sigma_{n})=1.$ On the other side, by the inequality \eqref{3}, we know that the sequence $\left\lbrace \sigma_{n} \right\rbrace $ is decreasing and hence we can apply Lemma \ref{l1} to get $\lim_{n\rightarrow \infty }  \sigma_{n}=0.$
	Now, we claim that $\left\{ x_{n}\right\} $ is a Cauchy sequence, for this, consider the condition $(\vartheta3).$ From $(\vartheta3),$ there exist $r \in (0,1)$ and $\lambda \in (0,\infty]$ such that
	\begin{equation}
	\underset{n\rightarrow \infty }{\lim } \frac{\vartheta (\sigma_{n})-1}{(\sigma_{n})^{r}}=\lambda.
	\end{equation}
	Take $\delta \in (0,\lambda).$ By the definition of limit, there exists $n_{0} \in \mathbb{N}$ such that
	\begin{equation*}
	[\sigma_{n}]^{r} \leq \delta^{-1} [\vartheta (\sigma_{n})-1], \ \ \ \ \text{for all} \ n>n_{0}.
	\end{equation*}
	Using (\ref{6}) and the above inequality, we deduce
	\begin{equation*}
	n[\sigma_{n}]^{r} \leq \delta^{-1} n([\vartheta (\sigma_{0})]^{k^{n}}-1), \ \ \ \ \text{for all} \ n>n_{0}.
	\end{equation*}
	This implies that
	\begin{align*}
	\underset{n\rightarrow \infty }{\lim } n[\sigma_{n}]^{r}=\underset{n\rightarrow \infty }{\lim } n[d(x_{n},x_{n+1})]^{r}=0.
	\end{align*}%
	Thence, there exists $n_{1} \in \mathbb{N}$ such that
	\begin{align}
	d(x_{n},x_{n+1}) \leq \frac{1}{n^{1/ r}}, \ \ \ \ \text{for all} \ n>n_{1}. \label{8}
	\end{align}%
	Let $m>n>n_{1}.$ Then, using the triangular inequality and (\ref{8}), we have
	\begin{align*}
	d(x_{n},x_{m}) \leq \sum\limits_{k=n}^{m-1} d(x_{k},x_{k+1}) \leq \sum\limits_{k=n}^{m-1} \frac{1}{k^{1/ r}} \leq \sum\limits_{k=n}^{\infty} \frac{1}{k^{1/ r}}
	\end{align*}%
	and hence $\left\{ x_{n}\right\} $ is a Cauchy sequence in $X.$ From the completeness of $(X,d)$, there exists $u \in X$ such that $x_{n} \to u$ as $n \to \infty.$ We now show that $u$ is a fixed point of $T.$ Suppose that $d(u,Tu)>0.$ Taking Remark \ref{rem} into account, we have
	\begin{align*}
	d(u,Tu) &\le d(u,x_{n+1})+d(x_{n+1},Tu) \\
	&\le d(u,x_{n+1})+H(Tx_{n},Tu)\\
	&\le d(u,x_{n+1})+ \rho(d(x_{n},u),d(x_{n},Tx_{n}),d(u,Tu),d(x_{n},Tu),d(u,Tx_{n}))\\
	&\le d(u,x_{n+1})+ \rho(d(x_{n},u),d(x_{n},x_{n+1}),d(u,Tu),d(x_{n},u)+d(u,Tu),d(u,x_{n+1})).
	\end{align*}
	Passing to limit as $n \to \infty$ in the above inequality, we obtain
	\begin{align*}
	d(u,Tu) \le \rho(0,0,d(u,Tu),0+d(u,Tu),0),
	\end{align*} 
	which implies by Lemma \ref{l2} that
	\begin{align*}
	0<d(u,Tu) <0,
	\end{align*}
	which is a contradiction. Hence $d(u,Tu)=0.$ Since $Tu$ is closed, we deduce that $u \in Tu.$
\end{proof}

In the next theorem, we replace $K(X)$ with $CB(X)$ by considering
an additional condition for the function $\vartheta$.

\begin{theorem}\label{t2}
	Let $(X,d)$ be a complete metric space and $T\colon X \to CB(X)$
	a $\vartheta_\rho$-contraction with right
	continuous function $\vartheta \in \Omega.$ Then $T$ has a fixed point.
\end{theorem}

\begin{proof}
	Let $x_{0}\in X$ and $x_{1}\in Tx_{0}.$ If $x_{0}=x_{1}$ or $x_{1}\in Tx_{1},$ then $x_{1}$ is a fixed point of $T.$ Herewith, we assume that $x_{0}\neq x_{1}$ and $x_{1}\notin Tx_{1}, \mbox{ and hence }d(x_{1},Tx_{1})>0$. From \eqref{a1}, we get
	\begin{align*}
	\vartheta(d(x_{1},Tx_{1}))&\le\vartheta(H(Tx_{0},Tx_{1}))\\
	&\leq [\vartheta(\rho(d(x_{0},x_{1}),d(x_{0},Tx_{0}),d(x_{1},Tx_{1}),d(x_{0},Tx_{1}),d(x_{1},Tx_{0})))]^{k}\\
	&\le [\vartheta(\rho(d(x_{0},x_{1}),d(x_{0},x_{1}),d(x_{1},Tx_{1}),d(x_{0},x_{1})+d(x_{1},Tx_{1}),0))]^{k},	
	\end{align*}
	and so 
	\begin{align*}
	d(x_{1},Tx_{1})<\rho(d(x_{0},x_{1}),d(x_{0},x_{1}),d(x_{1},Tx_{1}),d(x_{0},x_{1})+d(x_{1},Tx_{1}),0).	
	\end{align*}
	Then Lemma \ref{l2} gives that $d(x_{1},Tx_{1})<d(x_{0},x_{1}).$ Thus, we obtain
	\begin{align*}
	\vartheta(d(x_{1},Tx_{1}))&\le \vartheta(H(Tx_{0},Tx_{1}))\\
	&\le  [\vartheta(\rho(d(x_{0},x_{1}),d(x_{0},x_{1}),d(x_{1},Tx_{1}),d(x_{0},x_{1})+d(x_{1},Tx_{1}),0))]^{k} \\
	&<[\vartheta(\rho(d(x_{0},x_{1}),d(x_{0},x_{1}),d(x_{0},x_{1}),2d(x_{0},x_{1}),0))]^{k}\\
	&\le[\vartheta(d(x_{0},x_{1})\rho(1,1,1,2,0))]^{k}\\ &\le[\vartheta(d(x_{0},x_{1}))]^{k},
	\end{align*}
	and hence
	\begin{align*}
	\vartheta(H(Tx_{0},Tx_{1}))<[\vartheta(d(x_{0},x_{1}))]^{k}.
	\end{align*}
	By the property of right continuity of $\vartheta \in  \Omega,$  there exists a real number $h_{1}>1$ such that
	\begin{align}\label{th1}
	\vartheta(h_{1}H(Tx_{0},Tx_{1}))\leq [\vartheta(d(x_{0},x_{1}))]^{k}.
	\end{align}
	From 
	\[d(x_{1},Tx_{1})\le H(Tx_{0},Tx_{1})<h_{1}H(Tx_{0},Tx_{1}),\] 
	by Lemma \ref{lm1}, there exists $x_{2} \in Tx_{1}$ such that $d(x_{1},x_{2})\le h_{1}H(Tx_{0},Tx_{1}).$ Thus, by \eqref{th1}, we infer that
	\begin{align*}
	\vartheta(d(x_{1},x_{2})) \le \vartheta(h_{1}H(Tx_{0},Tx_{1}))\leq [\vartheta(d(x_{0},x_{1}))]^{k}.
	\end{align*}
	Continuing in this manner, we build two sequences $\left\lbrace
	x_{n} \right\rbrace \subset X$ and  $\left\lbrace h_{n}
	\right\rbrace \subset (1,\infty)$ such that $x_{n}\neq x_{n+1} \in
	Tx_{n}$ and
	\begin{align*}
	1<\vartheta(d(x_{n},x_{n+1}))\le \vartheta(h_{n} H(Tx_{n-1},Tx_{n}))\le[\vartheta(d(x_{n-1},x_{n}))]^{k},  \ \text{for all } n \in \mathbb{N}.
	\end{align*}
	Hence,
	\begin{equation*}
	1<\vartheta(d(x_{n},x_{n+1})) \le[\vartheta(d(x_{0},x_{1}))]^{k^{n}}, \ \  \text{for all } n \in \mathbb{N}.
	\end{equation*}
	which gives that
	\begin{equation*}
	\underset{n\rightarrow \infty }{\lim } \vartheta(d(x_{n},x_{n+1}))=1.
	\end{equation*}
	The rest of the proof is analogous with the proof of Theorem \ref{t1}.
\end{proof}
The following example illustrate Theorem \ref{t2} (resp. Theorem
\ref{t1}) where Theorem \ref{thvetro} is not
applicable.
\begin{example}\rm
	Let $X=[0,2]\cup \left\lbrace4,6,8,\ldots \right\rbrace $ be endowed with the metric 
	\begin{equation*}
	d\left( x,y\right) =\left\{
	\begin{array}{ll}
	0, &  x=y,  \\
	\left\vert x-y\right\vert , &   x,y \in [0,2],\\
	\max\{x,y\} , & \mbox{at least one of }  x,y \notin [0,2].
	\end{array}%
	\right.  
	\end{equation*}
	Then $(X,d)$ is a complete metric space. Define $T\colon X \to CB(X)$ by
	\begin{equation*}
	Tx=\left\{
	\begin{array}{ll}
	\{\frac{8}{9}\} , &  x=0, \\
	\left[ 0, 1\right]  , &   0<x \le 2,\\
	\{0,2,\ldots,x-2\}  , &   x \ge 4.
	\end{array}%
	\right.  
	\end{equation*}
	Letting $u(x,y):=\max\left\lbrace d(x,y),d(x,Tx),d(y,Ty),d(x,Ty),d(y,Tx)\right\rbrace .$ We claim that $T$ is a $\vartheta_\rho$-contraction with $\vartheta(t)=e^{\sqrt{te^{t}}}, \ k=e^{-1}\mbox{ and } \rho(x_{1},\ldots,x_{5})=\max\{x_{1}, \ldots,x_{5}\}.$ For that, we need to show that 
	\[\frac{H(Tx,Ty)}{u(x,y)}e^{H(Tx,Ty)-u(x,y)} \le e^{-2}, \quad\mbox{ for all } x,y \in X  \mbox{ with }  H(Tx,Ty)>0.\]
	Note that $H(Tx,Ty)>0 \mbox{ if and only if } (x,y) \notin \{(x,x)\colon x \in X\} \cup (0,2]\times (0,2].$ By the symmetry property of the metric, we have the following cases:\\
	\textbf{Case 1.} If $y=0 \mbox{ and } x \in (0,2],$ since
	\begin{align*}
	H\left( Tx,Ty\right)=\frac{1}{9}= \frac{1}{8}\cdot\frac{8}{9}=\frac{1}{8}d(0,T0)\le \frac{1}{8}u(x,y),
	\end{align*}
	then we have
	\[\frac{H(Tx,Ty)}{u(x,y)}e^{H(Tx,Ty)-u(x,y)} \le \frac{\frac{1}{8}u(x,y)}{u(x,y)}e^{-\frac{7}{8}u(x,y)}\le \frac{1}{8}<e^{-2}.\]\\
	\textbf{Case 2.} If $y=0 \mbox{ and } x = 4,$ then $H(Tx,Ty)=10/9 \mbox{ and } u(x,y)=4,$ and so
	\begin{align*}\frac{H(Tx,Ty)}{u(x,y)}e^{H(Tx,Ty)-u(x,y)} &\le \frac{10}{36} \ e^{-\frac{26}{9}}<e^{-2}
	\end{align*}
	\textbf{Case 3.} If $y=0 \mbox{ and } x > 4,$ then $H(Tx,Ty)=x-2 \mbox{ and } d(x,y)=x,$ and so
	\begin{align*}\frac{H(Tx,Ty)}{u(x,y)}e^{H(Tx,Ty)-u(x,y)} &\le \frac{H(Tx,Ty)}{d(x,y)}e^{H(Tx,Ty)-d(x,y)}\\
	&\le \frac{x-2}{x}\ e^{-2}<e^{-2}
	\end{align*}
	\textbf{Case 4.} If $y \in (0,2] \mbox{ and } x = 4,$ then $H(Tx,Ty)=1 \mbox{ and } u(x,y)=4,$ and so
	\begin{align*}\frac{H(Tx,Ty)}{u(x,y)}e^{H(Tx,Ty)-u(x,y)} &\le \frac{1}{4} \ e^{-3}<e^{-2}
	\end{align*}
	\textbf{Case 5.}  If $y \in (0,2] \mbox{ and } x > 4,$ then $H(Tx,Ty)=x-2 \mbox{ and } d(x,y)=x,$ and so it results as in Case 3.\\
	\textbf{Case 6.}  If $x>y\ge 4,$ then $H(Tx,Ty)=x-2 \mbox{ and } d(x,y)=x,$ and hence it follows as in Case 3.
	
	Consequently, all conditions of Theorem \ref{t2} (resp. Theorem \ref{t1}) are satisfied. Then $T$ has a fixed point in $X.$ Note that the set of fixed points of $T$ is not finite.
	
	On the other hand, for $y=0$ and $x=1/9,$ we get
	\begin{align*}
	\vartheta \left( H\left( Tx,Ty\right)\right)=\vartheta \left( H\left( T\frac{1}{9},T0\right)\right)=\vartheta \left( \frac{1}{9}
	\right) >\left[ \vartheta \left( \frac{1}{9}\right)\right] ^{k}=[\vartheta \left( d(x,y)\right)]^{k},
	\end{align*}
	for all $\vartheta \in  \Omega  \mbox{ and } k \in (0,1) .$
	Therefore, $T$ is not weak $\vartheta$-contraction and hence Theorem \ref{thvetro} can not applied to this example.
	
	Also, if $y=0$ and $x>4,$ then $H(Tx,T0)=x-2 \mbox{ and } u(x,0)=x,$ and hence
	\begin{equation*}
	\underset{x\rightarrow \infty }{\lim } \frac{H(Tx,T0)}{u(x,0)}=\underset{x\rightarrow \infty }{\lim } \frac{x-2}{x}=1.
	\end{equation*}
	That's why, we can not find $\lambda \in (0,1)$ such that $H(Tx,Ty) \le \lambda u(x,y).$
\end{example}
The following corollaries express us that we can obtain various types of contractive multivalued mappings by using $\vartheta_\rho$-contraction.
\begin{corollary} (\cite{nadler})
	Let $(X,d)$ be a complete metric space and $T\colon X \to CB(X)$ (resp. $K(X)$) a $\vartheta$-contraction of Nadler type, that is, there exist $\vartheta \in \Omega \mbox{ and } k\in(0,1) $ such that
	\begin{equation*} 
	\vartheta(H(Tx,Ty)) \leq [\vartheta(d(x,y)) ]^{k},\  \mbox{ for all } x,y \in X  \mbox{ with }  H(Tx,Ty)>0.
	\end{equation*}
	Then $T$ has a fixed point.
\end{corollary}
\begin{proof}
	Consider $\rho \in \mathcal{P}$ given by $\rho(x_{1},x_{2},x_{3},x_{4},x_{5})=x_{1}.$ Then $T$ is a $\vartheta_\rho$-contraction and the result follows from Theorem \ref{t2} (resp. Theorem \ref{t1}).
\end{proof}
\begin{corollary} (\cite{kannan})
	Let $(X,d)$ be a complete metric space and $T\colon X \to CB(X)$ (resp. $K(X)$) a $\vartheta$-contraction of Kannan type, that is, there exist $\vartheta \in \Omega \mbox{ and } k\in(0,1) $ such that
	\begin{equation*} 
	\vartheta(H(Tx,Ty)) \leq [\vartheta(d(x,Tx)+d(y,Ty)) ]^{k},\  \mbox{ for all } x,y \in X  \mbox{ with }  H(Tx,Ty)>0.
	\end{equation*}
	Then $T$ has a fixed point.
\end{corollary}
\begin{proof}
	Consider $\rho \in \mathcal{P}$ given by $\rho(x_{1},x_{2},x_{3},x_{4},x_{5})=x_{2}+x_{3}.$ Then $T$ is a $\vartheta_\rho$-contraction and the result follows from Theorem \ref{t2} (resp. Theorem \ref{t1}).
\end{proof}
\begin{corollary} (\cite{chat})
	Let $(X,d)$ be a complete metric space and $T\colon X \to CB(X)$ (resp. $K(X)$) a $\vartheta$-contraction of Chatterjea type, that is, there exist $\vartheta \in \Omega \mbox{ and } k\in(0,1)  $ such that
	\begin{equation*} 
	\vartheta(H(Tx,Ty)) \leq [\vartheta(d(x,Ty)+d(y,Tx) )]^{k},\  \mbox{ for all } x,y \in X  \mbox{ with }  H(Tx,Ty)>0.
	\end{equation*}
	Then $T$ has a fixed point.
\end{corollary}
\begin{proof}
	Consider $\rho \in \mathcal{P}$ given by $\rho(x_{1},x_{2},x_{3},x_{4},x_{5})=x_{4}+x_{5}.$ Then $T$ is a $\vartheta_\rho$-contraction and the result follows from Theorem \ref{t2} (resp. Theorem \ref{t1}).
\end{proof}
\begin{corollary} (\cite{reich})
	Let $(X,d)$ be a complete metric space and $T\colon X \to CB(X)$ (resp. $K(X)$) a $\vartheta$-contraction of Reich type, that is, there exist $\vartheta \in \Omega, \ k\in(0,1)$ and non-negative real numbers $ \alpha,\beta,\gamma \mbox{ with }  \alpha+\beta+\gamma\le 1 $ such that
	\begin{equation*} 
	\vartheta(H(Tx,Ty)) \leq [\vartheta(\alpha d(x,y)+\beta d(x,Tx)+\gamma d(y,Ty ))]^{k},
	\end{equation*}
	for all $x,y \in X$ with $H(Tx,Ty)>0.$ Then $T$ has a fixed point.
\end{corollary}
\begin{proof}
	Consider $\rho \in \mathcal{P}$ given by $\rho(x_{1},x_{2},x_{3},x_{4},x_{5})=\alpha x_{1}+ \beta x_{2}+\gamma x_{3}.$ Then $T$ is a $\vartheta_\rho$-contraction and the result follows from Theorem \ref{t2} (resp. Theorem \ref{t1}).
\end{proof}
\begin{corollary} (\cite{ber})
	Let $(X,d)$ be a complete metric space and $T\colon X \to CB(X)$ (resp. $K(X)$) a $\vartheta$-contraction of Berinde type, that is, there exist $\vartheta \in \Omega, \ k\in(0,1), \ \alpha\in (0,1]$ and $L\ge0 $ such that
	\begin{equation*} 
	\vartheta(H(Tx,Ty)) \leq [\vartheta(\alpha d(x,y)+L d(y,Tx))]^{k},
	\end{equation*}
	for all $x,y \in X$ with $H(Tx,Ty)>0.$ Then $T$ has a fixed point.
\end{corollary}
\begin{proof}
	Consider $\rho \in \mathcal{P}$ given by $\rho(x_{1},x_{2},x_{3},x_{4},x_{5})=\alpha x_{1}+ L x_{5}.$ Then $T$ is a $\vartheta_\rho$-contraction and the result follows from Theorem \ref{t2} (resp. Theorem \ref{t1}).
\end{proof}
\begin{corollary} (\cite{hardy})
	Let $(X,d)$ be a complete metric space and $T\colon X \to CB(X)$ (resp. $K(X)$) a $\vartheta$-contraction of Hardy-Rogers type, that is, there exist
	$\vartheta \in \Omega, \ k\in(0,1)$ and non-negative real numbers $ \alpha,\beta,\gamma, \delta,L \mbox{ with }  \alpha+\beta+\gamma+2\delta \le 1 $ such that
	\begin{equation*} 
	\vartheta(H(Tx,Ty)) \leq [\vartheta(\alpha d(x,y)+\beta d(x,Tx)+\gamma d(y,Ty )+\delta d(x,Ty )+L d(y,Tx ))]^{k},
	\end{equation*}
	for all $x,y \in X$ with $H(Tx,Ty)>0.$ Then $T$ has a fixed point.
\end{corollary}
\begin{proof}
	Consider $\rho \in \mathcal{P}$ given by $\rho(x_{1},x_{2},x_{3},x_{4},x_{5})=\alpha x_{1}+ \beta x_{2}+\gamma x_{3}+\delta x_{4}+L x_{5}.$ Then $T$ is a $\vartheta_\rho$-contraction and the result follows from Theorem \ref{t2} (resp. Theorem \ref{t1}).
\end{proof}
\begin{corollary} (\cite{cir})
	Let $(X,d)$ be a complete metric space and $T\colon X \to CB(X)$ (resp. $K(X)$) a $\vartheta$-contraction of \'{C}iri\'{c} type I, that is, there exist
	$\vartheta \in \Omega \mbox{ and } k\in(0,1)$ such that
	\begin{equation*} 
	\vartheta(H(Tx,Ty)) \leq [  \vartheta( \max\{d(x,y),d(x,Tx),d(y,Ty),\frac{1}{2}[d(x,Ty)+d(y,Tx)]\}  ) ] ^{k},
	\end{equation*}
	for all $x,y \in X$ with $H(Tx,Ty)>0.$ Then $T$ has a fixed point.
\end{corollary}
\begin{proof}
	Consider $\rho \in \mathcal{P}$ given by $\rho(x_{1},x_{2},x_{3},x_{4},x_{5})=\max\{x_{1},x_{2},x_{3},\frac{x_{4}+x_{5}}{2}\}.$ Then $T$ is a $\vartheta_\rho$-contraction and the result follows from Theorem \ref{t2} (resp. Theorem \ref{t1}).
\end{proof}
\begin{corollary} (\cite{ciric})
	Let $(X,d)$ be a complete metric space and $T\colon X \to CB(X)$ (resp. $K(X)$) a $\vartheta$-contraction of \'{C}iri\'{c} type II , that is, there exist
	$\vartheta \in \Omega \mbox{ and } k\in(0,1)$ such that
	\begin{equation*} 
	\vartheta(H(Tx,Ty)) \leq [\vartheta(\max\left\lbrace d(x,y),d(x,Tx),d(y,Ty),d(x,Ty),d(y,Tx)\right\rbrace )]^{k},
	\end{equation*}
	for all $x,y \in X$ with $H(Tx,Ty)>0.$ Then $T$ has a fixed point.
\end{corollary}
\begin{proof}
	Consider $\rho \in \mathcal{P}$ given by $\rho(x_{1},x_{2},x_{3},x_{4},x_{5})=\max\{x_{1},x_{2},x_{3},x_{4},x_{5}\}.$ Then $T$ is a $\vartheta_\rho$-contraction and the result follows from Theorem \ref{t2} (resp. Theorem \ref{t1}).
\end{proof}
\begin{corollary} (\cite{zam})
	Let $(X,d)$ be a complete metric space and $T\colon X \to CB(X)$ (resp. $K(X)$) a Zamfirescu type $\vartheta$-contraction, that is, there exist
	$\vartheta \in \Omega \mbox{ and } k\in(0,1)$ such that
	\begin{equation*} 
	\vartheta(H(Tx,Ty)) \leq [  \vartheta( \max\{d(x,y),\frac{1}{2}[d(x,Tx)+d(y,Ty)],\frac{1}{2}[d(x,Ty)+d(y,Tx)]\}  ) ] ^{k},
	\end{equation*}
	for all $x,y \in X$ with $H(Tx,Ty)>0.$ Then $T$ has a fixed point.
\end{corollary}
\begin{proof}
	Consider $\rho \in \mathcal{P}$ given by $\rho(x_{1},x_{2},x_{3},x_{4},x_{5})=\max\{x_{1},\frac{x_{2}+x_{3}}{2},\frac{x_{4}+x_{5}}{2}\}.$ Then $T$ is a $\vartheta_\rho$-contraction and the result follows from Theorem \ref{t2} (resp. Theorem \ref{t1}).
\end{proof}

\section{\textbf{An Application}}
First of all, we recall some basic definitions of fractional calculus (for more details, see \cite{abbas,kilbas}).
For a continuous function $ f:\mathbb{R^+} \to \mathbb{R}, $ the Caputo fractional derivative of order $ \beta $ is defined by
\[ {^C}D^\beta f(t)=\frac{1}{\Gamma(n-\beta)} \int_{0}^{t} (t-s)^{n-\beta-1} f^n(s)ds, \ \ n-1<\beta<n,  \ \ n = [\beta] + 1,\]
and the Riemann–Liouville fractional integral of the function $f$ of order $ \beta $ is given by
\[I^\beta f(t)=\frac{1}{\Gamma(\beta)} \int_{0}^{t}(t-s)^{\beta-1}f(s)ds, \ \ \beta>0, \]	provided the right hand-side is point-wise defined on $ \mathbb{R^+}  $, where $ \Gamma(\cdot) $ is the gamma function, which is defined by $\Gamma(\beta)= \int_{0}^{\infty} t^{\beta-1}e^{-t}dt. $

Let $X:=\mathcal{C}(J,\mathbb{R})$ be the Banach space of all continuous real valued functions defined on $J=[t_0,T]$ endowed with the norm defined by $ \left\|x\right\|=\sup\{\left| x(t)\right|:t\in J \}  . $ By $ \mathcal{L}^1(J,\mathbb{R}) ,$ we denote the Banach space of all measurable functions $ x : J \to \mathbb{R} $ which are Lebesgue integrable endowed
with the norm 
\[  \left\|x\right\|_{\mathcal{L}^1}=\int_{t_0}^{T} \left| x(t)\right|dt. \] 

A multivalued mapping $ F\colon J \to K(\mathbb{R})$ is called measurable if for every $y \in \mathbb{R},$ the function 
\[  t \to d(y,F(t))=\inf\{\left|y-z\right|:z \in F(t)\}   \]
is measurable. 

Let $ G\colon J \times \mathbb{R} \to K(\mathbb{R})$ be a multivalued map and $ u \in X, $ then the set of selections of $ G(\cdot,\cdot), $ denoted by $ S_{G,u}, $ is of lower semi-continuous type if\[ S_{G,u}=\{w\in\mathcal{L}^1(J,\mathbb{R}) \colon w(t)\in G(t,u(t)), \ \mbox{  for almost each } t\in J \} \]
is lower semi-continuous with nonempty closed and decomposable values.

In this section, we present an application of Theorem \ref{t1} in establishing
the existence of solutions for problem \eqref{frac}. To define the solution of problem \eqref{frac}, let us consider its linear variant given by
\begin{equation}\label{frac1}
\left\{
\begin{array}{ll} 
{^C}D^\beta_{t_0}x(t)= f(t), & t\in J, \vspace{0,3cm} \\
x^{(k)}(\alpha)=a_k+\int_{t_0}^{\alpha}g_k(s)ds, & k=0,1,\ldots,n-1, \ \delta \in J,
\end{array}%
\right.  
\end{equation}
where $f \in\mathcal{AC}(J,\mathbb{R}) \ (\mathcal{AC}(J,\mathbb{R})=\{f\colon J \to \mathbb{R}:f \mbox{ is absolutely continuous}  \} ) \mbox{ and } g_k \in X.$
\begin{lemma} (\cite{ahmad}) \label{lem}
	The fractional nonlocal boundary value problem \eqref{frac1} is equivalent to the integral equation
	\[ x(t)=I^\beta f(t)+\sum_{k=0}^{n-1} \frac{(t-\alpha)^k}{k!} \left(a_k+\int_{t_0}^{\alpha}g_k(s)ds-I^{\beta-k}f(\alpha) \right), \quad t \in J. \]
\end{lemma}
Our hypotheses are on the following data :
\begin{enumerate}[(A)]
	\item Let $ F\colon J \times \mathbb{R} \to K(\mathbb{R})$ be such that $ F(\cdot,x)\colon J  \to K(\mathbb{R})$ is measurable for each $x \in \mathbb{R};$\vspace{0,2cm}
	\item  for almost all $t\in J \mbox{ and }x,\tilde{x}\in \mathbb{R} \mbox{ with } m \in \mathcal{C}(J,(0,\infty))$
	\[ H(F(t,x),F(t,\tilde{x}))\le m(t) \left| x-\tilde{x}\right| \] 
	and $d(0,F(t,0))\le m(t);$\vspace{0,2cm}
	\item there exist functions $p_k\in \mathcal{C}(J,(0,\infty))$ such that \[\left| g_k(t,x)-g_k(t,\tilde{x})\right| \le p_k(t) \left| x-\tilde{x}\right|,  \]
	for $t\in J, \ k=0,1,\ldots,n-1 \mbox{ and }x,\tilde{x}\in \mathbb{R}; $\vspace{0,2cm}
	\item there exists $ \tau \in (0,\infty)$ such that  
	\[  \gamma_1 \left\| m\right\| +\gamma_2\le e^{-\tau},  \]
	where 
	\[ \gamma_1=\left\lbrace\frac{2}{\Gamma(\beta+1)}+\sum_{k=1}^{n-1} \frac{1}{k! \, \Gamma(\beta-k+1)} \right\rbrace (T-t_{0})^\beta \]
	and
	\[ \gamma_2= \sum_{k=0}^{n-1}\frac{(T-t_{0})^k\left\|p_k \right\| }{k!}.\]
\end{enumerate}
We are now ready to present main result of this section.
\begin{theorem}\label{thfrac}
Assume that the conditions $ (A)-(D) $ hold. Then the fractional differential
inclusion problem \eqref{frac} has at least one solution on $ X. $
\end{theorem}
\begin{proof}
Using Lemma \ref{lem}, define an operator $ \Lambda_F \colon X \to P(X) $ by
\begin{align*}
\Lambda_F(x)=&\left\lbrace v \in \mathcal{C}(J,\mathbb{R}) \colon v(t)= \int_{t_0}^{t}\frac{(t-s)^{\beta-1}}{\Gamma(\beta)}f(s)ds \right. \\
&+\sum_{k=0}^{n-1} \frac{(t-\alpha)^k}{k!} \left. \left(a_k+\int_{t_0}^{\alpha}g_k(s,x(s))ds-\int_{t_0}^{\alpha}\frac{(\alpha-s)^{\beta-k-1}}{\Gamma(\beta-k)}f(s)ds \right)\right\rbrace 
\end{align*}
for $ f\in S_{F,x}. $ Note that the set $S_{F,x}$ is nonempty for each $ x \in X $ by assumption $(A),$ so $F$ has a measurable selection (see Theorem 3.6 in \cite{hu}). Also, $\Lambda_F(x)$ is compact for each $ x \in X. $ This is obvious since $ S_{F,x}$ is compact
($ F $ has compact values), and therefore we omit its proof. We now prove that $\Lambda_F \mbox{ is a } \vartheta_\rho$-contraction. Let $x,\tilde{x} \in \mathcal{C}(J,\mathbb{R}) \mbox{ and } v_1 \in \Lambda_F(x). $ Then there exists $ w_1(t) \in F(t,x(t)) $ such that for all $ t\in J,$ we obtain 
\begin{align*}
v_{1}(t)=& \int_{t_0}^{t}\frac{(t-s)^{\beta-1}}{\Gamma(\beta)}w_{1}(s)ds  \\
&+\sum_{k=0}^{n-1} \frac{(t-\alpha)^k}{k!}  \left(a_k+\int_{t_0}^{\alpha}g_k(s,x(s))ds-\int_{t_0}^{\alpha}\frac{(\alpha-s)^{\beta-k-1}}{\Gamma(\beta-k)}w_{1}(s)ds \right).
\end{align*}
By the assumption $(B),$ we have
\[  H(F(t,x),F(t,\tilde{x}))\le m(t) \left| x(t)-\tilde{x}(t)\right|. \]
So, there exists $ h^{\star} \in F(t,\tilde{x}(t))$ such that
\[ \left|w_{1}(t)-h^{\star}\right| \le m(t) \left| x(t)-\tilde{x}(t)\right|, \quad t\in J. \]
Define the operator $ H\colon J \to P(\mathbb{R}) $ by
\[ H(t)=\{h^{\star}\in \mathbb{R}\colon\left|w_{1}(t)-h^{\star}\right| \le m(t) \left| x(t)-\tilde{x}(t)\right|\}. \]
Since $ H(t) \cap F(t,\tilde{x}(t)) $ is measurable (see Proposition 3.4 in \cite{hu}), there exists a function $w_{2}(t)$ which is a measurable selection for $ H. $ Hence, $ w_{2}(t)\in F(t,\tilde{x}(t)) $ and for all $ t\in J, $ 
\[ \left|w_{1}(t)-w_{2}(t) \right| \le m(t) \left| x(t)-\tilde{x}(t)\right|. \]
Now, we define
\begin{align*}
v_{2}(t)=& \int_{t_0}^{t}\frac{(t-s)^{\beta-1}}{\Gamma(\beta)}w_{2}(s)ds  \\
&+\sum_{k=0}^{n-1} \frac{(t-\alpha)^k}{k!}  \left(a_k+\int_{t_0}^{\alpha}g_k(s,\tilde{x}(s))ds-\int_{t_0}^{\alpha}\frac{(\alpha-s)^{\beta-k-1}}{\Gamma(\beta-k)}w_{2}(s)ds \right).
\end{align*}
It follows that, for all $ t\in J $ 
\begin{align*}
\left| v_{1}(t)-v_{2}(t)\right|\le &\int_{t_0}^{t}\frac{(t-s)^{\beta-1}}{\Gamma(\beta)}\left| w_{1}(s)-w_{2}(s)\right|ds  \\
&+\sum_{k=0}^{n-1} \frac{(t-\alpha)^k}{k!} \int_{t_0}^{\alpha}\frac{(\alpha-s)^{\beta-k-1}}{\Gamma(\beta-k)}\left| w_{1}(s)-w_{2}(s)\right|ds\\ 
&+\sum_{k=0}^{n-1} \frac{(t-\alpha)^k}{k!}\int_{t_0}^{\alpha}\left| g_k(s,x(s))-g_k(s,\tilde{x}(s))\right|ds\\
\le& \left\lbrace \left\lbrace\frac{1}{\Gamma(\beta+1)}+\sum_{k=0}^{n-1} \frac{1}{k! \, \Gamma(\beta-k+1)} \right\rbrace (T-t_{0})^\beta \left\|m \right\|  \right. \\
& +\sum_{k=0}^{n-1}\left. \frac{(T-t_{0})^k\left\|p_k \right\| }{k!}  \right\rbrace \left\|x-\tilde{x} \right\| ,
\end{align*}
and so
\begin{align*}
\left| v_{1}(t)-v_{2}(t)\right|
\le& \left\lbrace \left\lbrace\frac{2}{\Gamma(\beta+1)}+\sum_{k=1}^{n-1} \frac{1}{k! \, \Gamma(\beta-k+1)} \right\rbrace (T-t_{0})^\beta \left\|m \right\|  \right. \\
& +\sum_{k=0}^{n-1}\left. \frac{(T-t_{0})^k\left\|p_k \right\| }{k!}  \right\rbrace \left\|x-\tilde{x} \right\| .
\end{align*}
Thus, we obtain
\begin{align*}
\left\|  v_{1}-v_{2}\right\| \le (\gamma_{1}\left\|m \right\|+\gamma_{2})   \left\|x-\tilde{x} \right\| \le e^{-\tau} \left\|x-\tilde{x} \right\|.
\end{align*}
Now, by just interchanging the role of $ x \mbox{ and } \tilde{x} $, we reach to
\begin{align} \label{fr}
H(\Lambda_F(x),\Lambda_F(\tilde{x})) \le   e^{-\tau} \left\|x-\tilde{x} \right\| .
\end{align}
Consider $\rho \in \mathcal{P}$ and $\vartheta \in \Omega $ given by $\rho(x_{1},x_{2},x_{3},x_{4},x_{5})=x_{1}$ and $\vartheta(t)=e^{\sqrt t},$ respectively. Then, by (\ref{fr}), we infer 
\begin{align*}
e^{\sqrt {H(\Lambda_F(x),\Lambda_F(\tilde{x}))}} \leq e^{\sqrt {e^{-\tau} \left\|x-\tilde{x} \right\| }} \leq \left[ e^{\sqrt { \left\|x-\tilde{x} \right\| }}\right]^{k}, 
\end{align*}
which implies that
\begin{align*}
\vartheta (H(\Lambda_F(x),\Lambda_F(\tilde{x}))) 
\leq& \left[\vartheta(\rho(\left\|x-\tilde{x} \right\|,\left\|x-\Lambda_F(x) \right\|,\left\| \tilde{x}-\Lambda_F(\tilde{x})\right\| ,\right.\\ &\left. \left\| x-\Lambda_F(\tilde{x})\right\| ,\left\| \tilde{x}-\Lambda_F(x)\right\|))   \right]^{k}, 
\end{align*}
for all $x,\tilde{x} \in X,$ where $k=\sqrt {e^{-\tau}}.$ Since $\tau>0,$ then $k \in (0,1).$  This means that $\Lambda_F$ is a $ \vartheta_\rho$-contraction. Consequently, by Theorem \ref{t1},  $\Lambda_F$ has a fixed point $ x\in X $ which is a solution of the problem \eqref{frac}. 
\end{proof}
\begin{example}
Consider the fractional differential inclusion problem given by
\begin{equation}\label{exfr}
\left\{
\begin{array}{ll} 
{^C}D^{6.7}_{0}x(t)\in F(t,x(t)), & t\in [0,1], \vspace{0,3cm} \\
x^{(k)}(0.5)=1+\int\limits_{0}^{0.5}\frac{s^k}{3(k+1)}e^{-x(s)}ds, & k=0,1,\ldots,6, 
\end{array}%
\right.  
\end{equation}
where $ t_0=0, \ T=1, \ \beta=6.7, \ \alpha=0.5, \ a_k=1, \ g_k(t,x(t))=\frac{t^k}{3(k+1)}e^{-x(t)}$ and \vspace{0,2cm} \linebreak $F\colon [0,1]\times\mathbb{R} \to P(\mathbb{R}) $ is a multivalued mapping given by
$  F(t,x)=\left[ 0, \ \frac{t \left| x(t)\right| }{8\, (1+\left| x(t)\right|)}\right]. $ Note that
\[ t\to F(t,x)= \left[ 0, \ \frac{t \left| x(t)\right| }{8\, (1+\left| x(t)\right|)}\right]\]
is measurable for each $x \in \mathbb{R},$ since both the lower and upper functions are measurable on $[0,1]\times\mathbb{R}.$ Also
\[\left| g_k(t,x)-g_k(t,\tilde{x})\right| \le \frac{t^k}{3\,(k+1)} \left| e^{-x(t)}-e^{-\tilde{x}(t)}\right|.  \]
Here $ \, p_k(t)= \frac{t^k}{3\,(k+1)}$ and so $ \left\|p_k\right\|= \frac{1}{3\,(k+1)}, \mbox{ for } k=0,1,\ldots,6. $ On the other hand, we infer that
\[ \sup\{\left|y \right|:y \in F(t,x) \} \le \frac{t \left| x(t)\right| }{8\, (1+\left| x(t)\right|)}\le\frac{1}{8}, \mbox{ for each } (t,x)\in [0,1]\times\mathbb{R}, \]
and
\begin{align*}
H(F(t,x),F(t,\tilde{x}))) &=\left( \left[ 0, \ \frac{t \left| x(t)\right| }{8\, (1+\left| x(t)\right|)}\right], \left[ 0, \ \frac{t \left| \tilde{x}(t)\right| }{8\, (1+\left| \tilde{x}(t)\right|)}\right]\right)  \\ 
&\le \frac{t}{8}\left|x- \tilde{x}\right| .
\end{align*}
Here $ m(t)=\frac{t}{8} \ \mbox{with} \ \left\|m \right\|\approx 0.125. $ Besides, we find that
\[ \gamma_1=\frac{2}{\Gamma(7.7)}+\frac{1}{\Gamma(6.7)}+\frac{1}{\Gamma(5.7)}+\frac{1}{\Gamma(4.7)}+\frac{1}{\Gamma(3.7)}+\frac{1}{\Gamma(2.7)}+\frac{1}{\Gamma(1.7)}\approx2.07, \]
\[ \gamma_2=\frac{1}{3}+\frac{1}{6}+\frac{1}{9}\cdot\frac{1}{2}+\frac{1}{12}\cdot\frac{1}{6}+\frac{1}{15}\cdot\frac{1}{24}+\frac{1}{18}\cdot\frac{1}{120}+\frac{1}{21}\cdot\frac{1}{720}=0.5727, \]
and so
\[ \gamma_1 \left\| m\right\| +\gamma_2 \, \approx \, (2.07)\cdot(0.125)+0.5727=0.83145 \le e^{-\tau} \]
where $ \tau \in \left(0,\frac{1}{6} \right] . $ The compactness of $ F $ together with the above calculations lead to the existence of solution
of the problem \eqref{exfr} by Theorem \ref{thfrac}.
\end{example}
\section{\bf{Conclusion}}
In this paper, a new type of contractions has been proposed for multivalued mappings  by weakening the conditions on $\vartheta$ and by using auxilary functions. New fixed point theorems have been derived for multivalued mappings on complete metric spaces by means of this new class of contractions, which generalize the results in \cite{ber,chat,cir,ciric,hardy,jleli,kannan,nadler,reich,vetro,zam} and many others in the literature. To support of effectiveness and usability of new theory have been furnished several examples. Finally, sufficient conditions have been investigated to ensure the existence of solutions for the nonlocal integral boundary value problem of Caputo type fractional differential inclusions by using the results obtained herein.

\bigskip


\end{document}